\input ./style/arxiv-vmsta.cfg
\documentclass[numbers,compress,v1.0.1]{vmsta}
\usepackage{mathbh}
\usepackage{vtexurl}

\volume{2}
\issue{2}
\pubyear{2015}
\firstpage{173}
\lastpage{184}
\doi{10.15559/15-VMSTA30}% Updated by VTEXPTS2LaTeX.exe, 28.07.2015
\startlocaldefs
\newtheorem{thm}{Theorem}
\newtheorem{lemma}{Lemma}
\theoremstyle{definition}
\newtheorem{defin}{Definition}

\def\ind{\mathbh{1}} %\hspace{-2,9pt}\textnormal{I}

\allowdisplaybreaks
\endlocaldefs

\begin{document}
\begin{frontmatter}

\title{A L{u}ndberg-type inequality for an~inhomogeneous renewal risk model}

\author{\inits{I.M.}\fnm{Ieva Marija}\snm{Andrulyt\.{e}}}\email
{i.m.andrulyte@gmail.com}
\author{\inits{E.}\fnm{Emilija}\snm{Bernackait\.{e}}}\email
{emilija.bernackaite@mif.vu.lt}
\author{\inits{D.}\fnm{Dominyka}\snm{Kievinait\.{e}}}\email
{d.kievinaite@gmail.com}
\author{\inits{J.}\fnm{Jonas}\snm{\v{S}iaulys}\corref{cor1}}\email
{jonas.siaulys@mif.vu.lt}
\address{Faculty of Mathematics and Informatics, Vilnius University,
Naugarduko~24, Vilnius~LT-03225, Lithuania}
\cortext[cor1]{Corresponding author.}

\markboth{I.M. Andrulyt\.{e} et al.}{A L{u}ndberg-type inequality for an~inhomogeneous renewal risk model}

\begin{abstract}
We obtain a Lundberg-type inequality in the case of an inhomogeneous
renewal risk model. We consider the model with independent, but not
necessarily identically distributed, claim sizes and the
interoccurrence times. In order to prove the main theorem, we first
formulate and prove an auxiliary lemma on large values of a sum of
random variables asymptotically drifted in the negative direction.
\end{abstract}

\begin{keyword}
Inhomogeneous model \sep renewal model \sep Lundberg-type inequality
\sep exponential bound \sep ruin probability
\MSC[2010] 91B30 \sep 60G50
\end{keyword}

\received{19 June 2015}% Updated by VTEXPTS2LaTeX.exe, 28.07.2015 12:33
\revised{25 July 2015}% Updated by VTEXPTS2LaTeX.exe, 28.07.2015 12:33
\accepted{26 July 2015}% Updated by VTEXPTS2LaTeX.exe, 28.07.2015 12:33
\publishedonline{31 July 2015}
\end{frontmatter}

\section{Introduction}\label{i}

The classical risk model and the renewal risk model are two models that
are traditionally used to describe the nonlife insurance business. The
classical risk model was introduced by Lundberg and Cram\'{e}r about a
century ago (see \cite{cramer,Lundberg1,Lundberg2} for
the source papers and \cite{rolski} for the historical environment). In
this risk model, it is assumed that interarrival times are identically
distributed, exponential, and independent random variables. In 1957,
the Danish mathematician E. Sparre Andersen proposed the renewal risk
model to describe the surplus process of the insurance company. In the
renewal risk model, the claim sizes and the interarrival times are
independent, identically distributed, nonnegative random variables (see
\cite{andersen} for the source paper and \cite{thorin} for additional
details). In this paper, we assume that interoccurrence times and claim
sizes are nonnegative random variables (r.v.s) that are not necessarily
identically distributed. We call such a model the inhomogeneous model
and present its exact definition. It is evident that the inhomogeneous
renewal risk model reflects better the real insurance activities in
comparison with the classical risk model or with the renewal
(homogeneous) risk model.

\begin{defin}\label{dd}
We say that the insurer's surplus $U(t)$ varies according to
the inhomogeneous renewal risk model if
\vspace{-2mm}
\begin{align*}
U(t)=U(\omega,t)=x+ct-\sum
_{i=1}^{\varTheta(t)}Z_i
\vspace{-2mm}
\end{align*}
for all $t\geqslant0$. Here:
\begin{itemize}
\item$x\geqslant0$ is the initial reserve;
\item claim sizes $\{ Z_1,Z_2,\dots\}$ form a sequence of independent
(not ne\-ces\-sarily identically distributed) nonnegative r.v.s;
\item$c>0$ is the constant premium rate;
\item$\varTheta(t)=\sum_{n=1}^{\infty}\ind_{\lbrace T_n\leqslant
t\rbrace}=\sup\lbrace n\geqslant0:T_n\leqslant t\rbrace$ is the
number of claims in the interval $[0,t]$,
where $T_0=0$, $T_n=\theta_1+\theta_2+\cdots+\theta_n$, $n\geqslant1$,
and the interarrival times $\{\theta_1,\theta_2, \ldots\}$ are
independent (not ne\-ces\-sarily identically distributed), nonnegative, and nondegenerate at zero r.v.s;
\item the sequences $\lbrace Z_1,Z_2,\ldots\rbrace$ and $\lbrace\theta
_1,\theta_2,\ldots\rbrace$ are mutually independent.
\end{itemize}
\end{defin}

A typical path of the surplus process of an insurance company is shown
in Fig. \ref{f01}.

\begin{figure}[t!]
\includegraphics{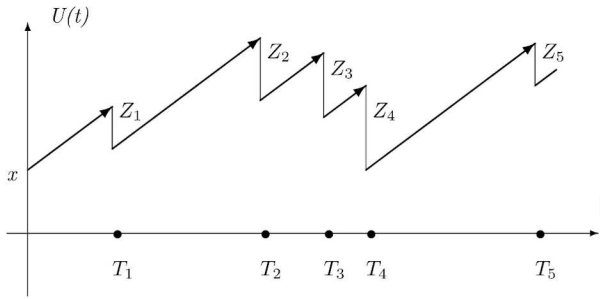}
\caption{Behavior of the surplus process.}\label{f01}
\end{figure}

If all claim sizes $\{Z_1,Z_2, \ldots\}$ and all interarrival times
$\{\theta_1,\theta_2,\ldots\}$ are identically distributed, then the
inhomogeneous renewal risk model becomes the homogeneous renewal risk model.

The \textit{time of ruin} and \textit{the ruin probability} are the
main critical characteristics of any risk model.
Let $\mathcal{B}$ denote the event of ruin. We suppose that
\begin{align}
\notag\mathcal{B}=\bigcup_{t\geqslant0}\bigl\lbrace\omega:U(
\omega ,t)<0\bigr\rbrace=\bigcup_{t\geqslant0} \Biggl\{
\omega:x+ct-\sum
_{i=1}^{\varTheta(t)}Z_i<0 \Biggr
\}\notag,
\end{align}
that is, that ruin occurs if at some time $t\geqslant0$ the surplus of
the insurance company becomes negative or, in other words, the insurer
becomes unable to pay all the claims.
The first time $\tau$ when the surplus drops to a level less than zero
is called \textit{the time of ruin}, that is, $\tau$ is the extended
r.v.\ for which
\begin{align*}
\notag\tau=\tau(\omega)= %
\begin{cases} \inf\lbrace t\geqslant0:U(\omega,t)<0\rbrace&\mbox{if }
\omega\in\mathcal{B},\\ \infty&\mbox{if } \omega\notin\mathcal{B}.
\end{cases} %
\notag
\end{align*}

\textit{The ruin probability} $\psi$ is defined by the equality
\begin{align}
\notag\psi(x)=\mathbb{P}(\mathcal{B})=\mathbb{P}(\tau=\infty).
\end{align}
Usually, we suppose that the main parameter of the ruin probability is
the initial reserve $x$, though actually the ruin probability, together
with time of ruin, depends on all components of the renewal risk model.

All trajectories of the process $U(t)$ are increasing functions
between times $T_n$ and $T_{n+1}$ for all $n=0,1,2,\ldots$\,. Therefore,
the random variables $U(\theta_1+\theta_2+\cdots+\theta_n)$, $n\geqslant
1$, are the local minimums of the trajectories. Consequently, we can
express the ruin probability in the following way (for details, see
\cite{embrechts} or \cite{mikosch}):
\begin{align*}
\psi(x)&=\mathbb{P} \Bigl(\inf_{n\in\mathbb{N}}U(\theta_1+\theta _2+\cdots+\theta_n)<0 \Bigr)\\
&=\mathbb{P} \Biggl(\inf_{n\in\mathbb{N}} \Biggl\lbrace x + c(\theta_1+\theta_2+\cdots+\theta_n)
-
\sum_{i=1}^{\varTheta(\theta_1+\cdots+\theta_n)}\!Z_i\Biggr\rbrace<0 \Biggr)
\\
%&=\mathbb{P}\bigg(\inf_{n\in\mathbb{N}}\bigg\lbrace x + c(\theta_1+
%\theta_2+...+\theta_n)-\sum_{i=1}^{n}Z_i\bigg\rbrace<0\bigg)\notag\\
&=\mathbb{P} \Biggl(\inf_{n\in\mathbb{N}}
\Biggl \lbrace x - \sum_{i=1}^{n} (Z_i-c\theta_{i})
\Biggr\rbrace <0 \Biggr)
\\
% &=\mathbb{P}\bigg(\inf_{n\in\mathbb{N}}\bigg\lbrace- \sum
%_{i=1}^{n} (Z_i-c\theta_{i})\bigg\rbrace<-x\bigg)
& = \mathbb{P} \Biggl(\sup_{n\in\mathbb{N}}
\Biggl\lbrace\sum_{i=1}^n (Z_i-c\theta_{i})
\Biggr\rbrace>x \Biggr).
\end{align*}

Further, in this paper, we restrict our study to the so-called \textit
{Lundberg-type inequa\-lity}. An exponential bound for the ruin
probability is usually called a~ Lundberg-type inequality. We further
give the well-known exponential bound for $\psi(x)$ in homogeneous
renewal risk model (see, for instance, Chapters ``Lundberg Inequality
for Ruin Probability\xch{'',}{,''} ``Collective Risk Theory\xch{'', }{,''}\xch{``}{``"}Adjustment
Coefficient,'' or ``Cramer--Lundberg Asymptotics'' in \cite{enciklopedija}).

\begin{thm}
Let the net profit condition
$\mathbb{E}Z_1-c\mathbb{E}\theta_1<0$ hold, and let
\hbox{$\mathbb{E}{\rm e}^{hZ_1}<\infty$} for some
$h>0$ in the homogeneous renewal risk model. Then, there is a number
$H>0$ such that
\begin{equation}
\label{aa} \psi(x)\leqslant{\rm e}^{-Hx}
\end{equation}
for all $x\geqslant0$.
If $\mathbb{E}{\rm e}^{R(Z_1-c\theta_1)}=1$
for some positive $R$, then we can take $H=R$ in
\eqref{aa}.\looseness=-1
\end{thm}

There exist a lot of different proofs of this theorem. The main ways
to prove inequality~\eqref{aa} are described in Chapter ``Lundberg
Inequality for Ruin Probability'' of \cite{enciklopedija}. Details of
some existing proofs were given, for instance, by Asmussen and
Albrecher \cite{asmussen}, Embrechts, Kl\"{u}ppelberg, and Mikosch \cite
{embrechts}, Embrechts and Veraverbeke~\cite{embrechts1}, Gerber \cite
{gerber}, and Mikosch \cite{mikosch}. We note only that the bound \eqref
{aa} can be proved using the exponential tail bound of Sgibnev~\cite
{sgibnev} and the inequality $\psi(0)<1$.

The following theorem is the main statement of the paper.

\begin{thm}\label{t}
Let the claim sizes $\{ Z_1,Z_2,\dots\}$ and the interarrival
times $\{\theta_1,\theta_2,\dots\}$ form an inhomogeneous renewal risk
model described in Definition \ref{dd}. Further, let the following
three conditions be satisfied:
\begin{align*}
(\mathcal{C}1)\hspace{0.1cm}
&
\sup_{i\in\mathbb{{N}}}
\mathbb{E} {\rm e}^{\gamma Z_{i}}<\infty\quad\mbox{for some}\ \gamma>0,
\\
(\mathcal{C}2)\hspace{0.1cm}
&
\lim_{u\to\infty}\sup_{i\in\mathbb{N}}
\mathbb{E}\big(\theta_{i}\ind_{\lbrace\theta_{i}>u\rbrace}\big)=0,
\\
(\mathcal{C}3)\hspace{0.1cm}
&
\limsup_{n\to\infty} \frac{1}{n}
\sum _{i=1}^{n}(\,\mathbb{E}Z_i-c\mathbb{E}\theta_{i}) <0.
\end{align*}
Then, there are constants $c_1>0$ and $c_2\geqslant0$ such that
$\psi(x)\leqslant{\rm e}^{-c_{1}x}$ for all $ x\geqslant c_2$.
\end{thm}

The inhomogeneous renewal risk model differs from the homogeneous one
because the independence and/or homogeneous distribution of sequences
of random variables $\lbrace Z_1,Z_2,\dots\rbrace$ and/or $\lbrace\theta
_1,\theta_2,\dots\rbrace$ are no longer required. The changes depend on
how the inhomogeneity in a particular model is understood. In
Definition \ref{dd}, we have chosen one of two possible directions used
in numerous articles that deal with inhomogeneous renewal risk models.
This is due to the fact that an inhomogeneity can be considered as the
possibility to have either differently distributed or dependent r.v.s
in the sequences.

The possibility to have differently distributed r.v.s was considered,
for instance, in \cite{ej,kej,lefevre,raducan}. In the first three works, the discrete-time inhomogeneous
risk model was considered. In such a model, the interarrival times are
fixed, and the claims $\lbrace Z_1,Z_2,\dots\rbrace$ are independent,
not necessary identically distributed, integer valued r.v.s. In \cite
{raducan}, the authors consider the model where the interarrival times
are identically distributed and have a particular distribution, whereas
the claims are differently distributed with distributions belonging to
a particular class. Bernackait\.{e} and \v{S}iau\-lys \cite{ej2} deal with an inhomogeneous renewal risk model where the r.v.s
$\lbrace\theta_1,\theta_2,\dots\rbrace$ are not necessarily identically
distributed, but the claim sizes $\lbrace Z_1,Z_2,\dots\rbrace$ have a
common distribution function. In this article, we consider a more
general renewal risk model. In the main theorem, we assume that not
only r.v.s $\lbrace\theta_1,\theta_2,\dots\rbrace$ are not necessarily
identically distributed, but also the same holds for the sequence of
claim sizes $\lbrace Z_1,Z_2,\dots\rbrace$.

There is another approach to inhomogeneous renewal risk models, which
implies the possibility to have dependence in sequences and mainly
found in works by Chinese researchers. In this kind of models, the
sequences $\lbrace Z_1,Z_2,\dots\rbrace$ and $\lbrace\theta_1,\theta
_2,\dots\rbrace$ consist of identically distributed r.v.s, but there
may be some kind of dependence between them. Results for such models
can be found, for instance, in \cite{chen} and \cite{wang}. Another
interpretation of dependence is also possible, where r.v.s in both
sequences $\lbrace Z_1,Z_2,\dots\rbrace$ and $\lbrace\theta_1,\theta
_2,\dots\rbrace$ still remain independent. Instead of that, the mutual
independence of these two sequences is no longer required. The idea of
this kind of dependence belongs to Albrecher and Teugels \cite
{albrecher}, and this encouraged Li, Tang, and Wu \cite{li} to study
renewal risk models having this dependence structure.

The rest of the paper consists of two sections. In Section \ref{au}, we
formulate and prove an auxiliary lemma. The proof of the main theorem
is presented in Section \ref{pa}.

\section{Auxiliary lemma}\label{au}

In order to prove the main theorem, we need an auxiliary lemma. In
Lemma~\ref{lll}, the conditions for r.v.s $\eta_{1}$, $\eta_{2}$, $\eta
_{3},\ldots$ are taken from articles by Smith \cite{smith} and
Bernac\-kait\. e and \v Siaulys \cite{ej2}.

\begin{lemma}\label{lll}
Let $\eta_{1}$, $\eta_{2}$, $\eta_{3},\ldots$ be independent
r.v.s such that
\begin{align*}
\bigl(\mathcal{C}1^*\bigr)\hspace{0.1cm}&
\sup_{i\in\mathbb{N}}\mathbb{E} {\rm e}^{\delta\eta_{i}}<\infty\quad \mbox{for some}
\ \delta >0,
\\
\bigl(\mathcal{C}2^*\bigr)\hspace{0.1cm}&
\lim_{u\to\infty}
\sup_{i\in\mathbb{N}}
\mathbb{E}\bigl(|\eta_{i}|\ind_{\lbrace\eta_{i}\leqslant-u\rbrace}\bigr)=0,
\\
\bigl(\mathcal{C}3^*\bigr)\hspace{0.1cm}&
\limsup_{n\to\infty}\frac{1}{n}
\sum_{i=1}^{n} \mathbb{E}\eta_{i} < 0.
\end{align*}
Then, there are constants $c_{3}>0$ and $c_{4}>0$ such that
\begin{align*}
\mathbb{P} \Biggl(\sup_{k\geqslant1}\sum
_{i=1}^k \eta _{i}>x \Biggr)\leqslant
c_{3}{\rm e}^{-c_{4}x}
\end{align*}
for all $x\geqslant0$.
\end{lemma}

\begin{proof}

First, we observe that, for all $x\geqslant0$,
\begin{align}
\mathbb{P}
\Biggl(\sup_{k\geqslant1}\sum_{i=1}^k \eta _{i}>x \Biggr)
&{}=\mathbb{P}
\Biggl(\bigcup_{k=1}^\infty \Biggl\lbrace \sum_{i=1}^k \eta_{i}>x \Biggr\rbrace \Biggr)\nonumber\\
\label{r0} &{}\leqslant \sum
_{k=1}^\infty
\mathbb{P} \Biggl(\sum
_{i=1}^k
\eta_{i}>x \Biggr).
\end{align}

By Chebyshev's inequality, for all $x\geqslant0$, $ 0<y\leqslant\delta
$, and $k\in\mathbb{N}$, we have
\begin{align}
\mathbb{P}
\Biggl(\sum_{i=1}^k\eta_{i}>x \Biggr)
&{}=\mathbb{P}
\bigl({\rm e}^{y\sum_{i=1}^k \eta_{i}}>{\rm e}^{yx} \bigr)\nonumber\\
&
{}\leqslant
{\rm e}^{-yx}\prod _{i=1}^{k} \mathbb{E} {\rm e}^{y\eta_{i}}.
\label{r1}
\end{align}

Moreover, for all $i\in\mathbb{N}$, $ 0<y\leqslant\delta$, and $u>0$,
we have
\begin{align}
&\label{r2} \mathbb{E} {\rm e}^{y\eta_{i}}=1+y\mathbb{E}
\eta_{i}+\mathbb{E}\bigl({\rm e}^{y\eta_{i}}-1-y\eta_{i}
\bigr)
\end{align}
and\vspace*{-2pt}
\begin{align*}
&\mathbb{E}\bigl({\rm e}^{y\eta_{i}}-1-y\eta_{i}\bigr)
\\[1pt]
&\quad{}
=
\mathbb{E}\bigl(\bigl({\rm e}^{y\eta_{i}}-1\bigr)\ind_{\lbrace\eta_{i}\leqslant -u\rbrace}\bigr)
-
y\mathbb{E}(\eta_{i}\ind_{\lbrace\eta_{i}\leqslant-u\rbrace
})
\\[1pt]
&\qquad{}
+
\mathbb{E}\bigl(\bigl({\rm e}^{y\eta_{i}}-1-y\eta_{i}\bigr)
\ind_{\lbrace-u<\eta _{i}\leqslant0\rbrace}\bigr)
+
\mathbb{E}\bigl(\bigl({\rm e}^{y\eta_{i}}-1-y
\eta_{i}\bigr)\ind _{\lbrace\eta_{i}>0\rbrace}\bigr).
\end{align*}\goodbreak

In order to evaluate the absolute value of the remainder term in \eqref
{r2}, we need the following inequalities:\vspace*{-2pt}
\begin{align*}
\bigl|{\rm e}^{v}-1\bigr|\leqslant|v|,&{}\quad v\leqslant0,
\\[1pt]
\bigl|{\rm e}^{v}-v-1\bigr|\leqslant\frac{v^2}{2},&{}\quad v\leqslant0,
\\[1pt]
\bigl|{\rm e}^{v}-v-1\bigr|\leqslant\frac{v^2}{2}{\rm e}^{v},&{}\quad v\geqslant0.
\end{align*}

Using them, we get\vspace*{-2pt}
\begin{align}
&\bigl|\mathbb{E}\bigl({\rm e}^{y\eta_{i}}-1-y\eta_{i}\bigr)\bigr|\nonumber\\[1pt]
&\quad{}\leqslant
2y\mathbb{E}\bigl(|\eta_{i}|\ind_{\lbrace\eta_{i}\leqslant-u\rbrace}\bigr)
+
\frac{y^{2}}{2} \mathbb{E}\bigl(\eta_{i}^2\ind_{\lbrace-u<\eta_{i}\leqslant0\rbrace}\bigr)
+
\frac{y^{2}}{2}\mathbb{E}\bigl(\eta_{i}^{2}{\rm e}^{y\eta_{i}}\ind_{\lbrace\eta_{i}>0\rbrace}\bigr)\nonumber\\
&\quad{}\leqslant
2y\sup_{i\in\mathbb{N}}\mathbb{E}\bigl(|\eta_{i}|\ind_{\lbrace\eta_{i}\leqslant-u\rbrace}\bigr)
+
\frac{y^{2}u^{2}}{2}
+
\frac{y^{2}}{2}\sup_{i\in\mathbb{N}}\mathbb {E}\bigl(\eta_{i}^{2}{\rm e}^{y\eta_{i}}\ind_{\lbrace\eta_{i}>0\rbrace}\bigr),
\label{r3}
\end{align}
where $i\in\mathbb{N}$, $ 0<y\leqslant\delta$, and $ u>0$.

Since\vspace*{-2pt}
\begin{align*}
\lim_{v\to\infty}\frac{{\rm e}^{\delta v/2}}{v^2}=\infty,
\end{align*}
we have\vspace*{-2pt}
\begin{align*}
{\rm e}^{{\delta}v/{2}}\geqslant v^2
\end{align*}
for all $v\geqslant c_5$, where $c_5=c_5(\delta)>0$.

Therefore,\vspace*{-2pt}
\begin{align}
&
\sup_{i\in\mathbb{N}}\mathbb{E}\bigl(\eta_{i}^{2}{\rm e}^{{\delta}\eta_{i}/2}\ind_{\lbrace\eta_{i}>0\rbrace}\bigr)
\nonumber\\
&\quad{}\leqslant
\sup_{i\in\mathbb{N}}\mathbb{E}\bigl(\eta_{i}^{2}{\rm e}^{{\delta}\eta_{i}/2}\ind _{\lbrace0<\eta_{i}\leqslant c_5\rbrace}\bigr)+\sup_{i\in\mathbb{N}}
\mathbb {E}\bigl(\eta_{i}^{2}{\rm e}^{\delta\eta_{i}/2}\ind_{\lbrace\eta_{i}>c_5\rbrace}\bigr)\nonumber\\
&\quad{}\leqslant
\bigl(c_5^{2}+1\bigr)\sup_{i\in\mathbb{N}}\mathbb{E} {\rm e}^{\delta\eta_{i}}<\infty.
\label{r:septinta}
\end{align}

Choosing $u=\frac{1}{\sqrt[4]{y}}$ in \eqref{r3} and using \eqref{r:septinta},
we get\vspace*{-2pt}
\begin{align}
&\bigl|\mathbb{E}\bigl({\rm e}^{y\eta_{i}}-1-y\eta_{i}\bigr)\bigr|\nonumber\\
&\quad{}\leqslant
2y\sup_{i\in\mathbb{N}}\mathbb{E}\Bigl(|\eta_{i}|
\ind_{\lbrace\eta_{i}\leqslant-\frac{1}{\sqrt[4]{y}}\rbrace}\Bigr)
+
\frac{y^{\frac{3}{2}}}{2}
+
\frac{y^{2}}{2}\sup_{i\in\mathbb{N}}
\mathbb{E}\bigl(\eta _{i}^{2}{\rm e}^{y\eta_{i}}\ind_{\lbrace\eta_{i}>0\rbrace}\bigr)\nonumber\\
&\quad{}\leqslant
y \biggl(2\sup_{i\in\mathbb{N}}\mathbb{E}
\bigl(|\eta_{i}|\ind _{\lbrace\eta_{i}\leqslant-\frac{1}{\sqrt[4]{y}}\rbrace}\bigr)
+
\frac{y^{\frac{1}{2}}}{2}
+
\frac{y}{2}\bigl(c_5^{2}+1\bigr)\sup_{i\in\mathbb{N}}\mathbb{E} {\rm e}^{\delta\eta_{i}} \biggr)\nonumber\\
&\quad{}=:
y\alpha(y),
\label{r:astunta}
\end{align}
where $i \in\mathbb{N}$, $y\in(0,\delta/2]$, $ c_5=c_5(\delta)$, and
\[
\alpha(y)=2\sup_{i\in\mathbb{N}}\mathbb{E}\Bigl(|
\eta_{i}|\ind _{\lbrace\eta_{i}\leqslant-\frac{1}{\sqrt[4]{y}}\rbrace}\Bigr)+\frac{y^{\frac
{1}{2}}}{2}+\frac{y}{2}
\bigl(c_5^{2}+1\bigr)\sup_{i\in\mathbb{N}}
\mathbb{E} {\rm e}^{\delta\eta_{i}}.
\]

Conditions ($\mathcal{C}1^*$) and ($\mathcal{C}2^*$) imply that $\alpha
(y)\downarrow0$ as $ y\rightarrow0$.

For an arbitrary positive $v$, we have
\begin{align*}
\sup
_{i\in\mathbb{N}}\mathbb{E} \bigr(|\eta_i|\ind_{\{\eta_i<0\}
} \bigr)&{}=
\sup
_{i\in\mathbb{N}}\mathbb{E} \big(|\eta_i|\ind_{\{
-v<\eta_i<0\}}+|
\eta_i|\ind_{\{\eta_i\leqslant-v\}} \big)
\\[1pt]
&{}\leqslant v+\sup
_{i\in\mathbb{N}}\mathbb{E} \big(|\eta_i|\ind
_{\{\eta_i\leqslant-v\}} \big).
\end{align*}

So, condition ($\mathcal{C}2^*$) implies that
\begin{equation}
\label{pap} \sup
_{i\in\mathbb{N}}\mathbb{E} \big(|\eta_i|
\ind_{\{\eta_i<0\}
} \big)<\infty.
\end{equation}

Denote
\[
\widehat{y}=\min \Bigl\{\delta/2, 1/ \Bigl(2\sup
_{i\in\mathbb
{N}}\mathbb{E} \big(|
\eta_i|\ind_{\{\eta_i<0\}} \big) \Bigr) \Bigr\}.
\]

If $y\in(0,\widehat{y}\,]$, then
\begin{align*}
y\bigl(\mathbb{E}\eta_i+\alpha(y)\bigr)
&{}>
{y}\mathbb{E}\eta_i
\\[2pt]
&{}=
y \mathbb{E} \bigl(\eta_i\ind_{\{\eta_i\geqslant0\}}+\eta_i\ind_{\{\eta_i<0\}} \bigr)
\\[2pt]
&{}\geqslant
y \mathbb{E}\bigl(\eta_i\ind_{\{\eta_i<0\}}\bigr)
\\[2pt]
&{}\geqslant
\widehat{y}\inf_{i\in\mathbb{N}}\mathbb{E} \bigl(\eta_i\ind _{\{\eta_i<0\}}\bigr)
\\[2pt]
&{}= - \widehat{y}\sup_{i\in\mathbb{N}}\mathbb{E} \big(|\eta_i|\ind_{\{\eta_i<0\}} \big)
\\[2pt]
&{}\geqslant -1/2
\end{align*}
for all $i\in\mathbb{N}$.

Therefore, \eqref{r1}, \eqref{r2}, \eqref{r:astunta}, and the
well-known inequality
\[
\ln(1+u)\leqslant u,\quad u>-1,
\]
imply that\vadjust{\eject}
\begin{align}
\mathbb{P} \Biggl(\sum
_{i=1}^k
\eta_{i}>x \Biggr) &{}\leqslant{\rm e}^{-yx}\prod
_{i=1}^{k}\bigl(1+y\mathbb{E}\eta_{i}+
\mathbb {E}\bigl({\rm e}^{y\eta_{i}}-1-y\eta_{i}\bigr)\bigr)\notag
\\[-2.5pt]
&{}\leqslant{\rm e}^{-yx}\prod_{i=1}^{k}
\bigl(1+y\bigl(\mathbb{E}\eta_{i}+\alpha(y)\bigr)\bigr) \notag
\\[-2.5pt]
&{}=\exp \Biggl\lbrace-yx+\sum_{i=1}^{k}
\ln\bigl(1+y\bigl(\mathbb{E}\eta_{i}+\alpha (y)\bigr)\xch{\bigr)\Biggr\rbrace}{\Biggr\rbrace} \notag
\\[-2.5pt]
&{}\leqslant \exp \Biggl\lbrace-yx+y\sum_{i=1}^{k}
\mathbb{E}\eta _{i}+yk\alpha(y) \Biggr\rbrace,\label{r:devinta}
\end{align}
where $k\in\mathbb{N}$, $ x\geqslant0$, and
$y\in(0,\widehat{y}\,]$.

By estimate \eqref{pap} and condition ($\mathcal{C}3^*$) we can suppose that
\begin{align*}
\notag\limsup_{n\to\infty}\frac{1}{n} \sum
_{i=1}^{n}
\mathbb {E}\eta_{i} = -c_{6}
\end{align*}
for some positive constant $c_6$. Then we have
\begin{align}
\notag\frac{1}{k}\sum_{i=1}^{k}
\mathbb{E}\eta_{i}\leqslant-\frac{c_{6}}{2}
\end{align}
for $k\geqslant M+1$ with some $M\geqslant1$.
Moreover, there exists $ y^{*}\in(0,\widehat{y}\,]$ such that $ \alpha
(y^{*})\leqslant{c_{6}}/{4}$ since $\alpha(y)\downarrow0$ as $
y\rightarrow0$.

Using results from \eqref{r0}, \eqref{r1}, and \eqref{r:devinta}, we derive
\begin{align*}
&\mathbb{P} \Biggl(\sup_{k\geqslant1}\sum
_{i=1}^k \eta _{i}>x \Biggr)
\\[-2.5pt]
&\quad{}\leqslant\sum_{k=1}^M \mathbb{P}
\Biggl(\sum_{i=1}^k \eta_{i}>x
\Biggr)+\sum_{k=M+1}^\infty\mathbb{P} \Biggl(\sum
_{i=1}^k\eta_{i}>x \Biggr)\notag
\\[-2.5pt]
&\quad{}\leqslant\sum_{k=1}^M {\rm
e}^{-y^{*}x}\prod_{i=1}^{k}\mathbb {E}
{\rm e}^{y^{*}\eta_{i}}+\sum_{k=M+1}^\infty
\mathbb{P} \Biggl(\sum_{i=1}^k
\eta_{i}>x \Biggr)\notag
\\[-2.5pt]
&\quad{}\leqslant\sum_{k=1}^M {\rm
e}^{-y^{*}x}\prod_{i=1}^{k}\mathbb {E}
{\rm e}^{y^{*}\eta_{i}}+\sum_{k=M+1}^\infty{\rm
e}^{-y^{*}x+y^{*}\sum_{i=1}^{k}\mathbb{E}\eta_{i}+y^{*}k\alpha
(y^{*})}\notag\\
&\quad{}\leqslant {\rm e}^{-y^{*}x} \Biggl(\sum_{k=1}^M
\prod_{i=1}^{k}\mathbb {E} {\rm
e}^{y^{*}\eta_{i}}+\sum_{k=0}^\infty{\rm
e}^{-ky^*{c_{6}}/{4}} \Biggr)\notag
\\[-2.5pt]
&\quad{}\leqslant {\rm e}^{-y^{*}x} \Biggl(\sum_{k=1}^M
\prod_{i=1}^{k}\varDelta +\frac{1}{1-{\rm e}^{-y^{*}c_6/4}}
\Biggr)\notag
\\[-2.5pt]
&\quad{}={\rm e}^{-y^{*}x} \biggl(\frac{\varDelta(\varDelta^M-1)}{\varDelta-1}+\frac
{{\rm e}^{y^{*}c_6/4}}{{\rm e}^{{y^{*}c_{6}/4}}-1}
\biggr)=:c_{3}{\rm e}^{-c_{4}x},\notag
\end{align*}
where
\begin{align*}
&x\geqslant0,\\
&\varDelta=1+\sup_{i\in\mathbb{N}}\mathbb{E} {\rm e}^{\delta\eta_{i}},\\
&c_{3}=\frac{\varDelta(\varDelta^M-1)}{\varDelta-1}+\frac{e^{y^{*}{c_{6}}/{4}}}{e^{y^{*}{c_{6}}/{4}}-1},\\
&c_{4}=y^{*}\in(0,\widehat{y}\,]
\end{align*}
with $M\geqslant1$, $c_6>0$, and $\widehat{y}>0$ defined previously.
The lemma is now proved.
\end{proof}

\section{Proof of Theorem \ref{t}}\label{pa}

In this section, we prove Theorem \ref{t}.

\begin{proof}
Since
\[
\psi(x)= \mathbb{P} \Biggl(\sup_{n\geqslant1} \Biggl
\lbrace \sum_{i=1}^n (Z_i-c
\theta_{i}) \Biggr\rbrace>x \Biggr),
\]
the desired bound of Theorem \ref{t} can be derived from auxiliary
Lemma \ref{lll}.

Namely, supposing that r.v.s $Z_i-c\theta_i$, $i\in\{1,2,\ldots\}$,
satisfy all conditions of Lemma \ref{lll}, we get
\[
\psi(x)\leqslant c_{7}{\rm e}^{-c_{8}x}
\]
for all $x\geqslant0$ with some positive $c_7$, $c_8$ independent of $x$.

Therefore,
\[
\psi(x)\leqslant c_{7}{\rm e}^{-c_8 x/2}{\rm e}^{-c_8 x/2}
\leqslant{\rm e}^{-c_8 x/2},
\]
with $x\geqslant\max\{0,(2\ln c_7)/c_8\}$,

Thus, it suffices to check all three assumptions in our lemma with
random variables $Z_i-c\theta_i,\, i\in\mathbb{N}$.
The requirement ($\mathcal{C}3^*$) of Lemma \ref{lll} is evidently
satisfied by condition ($\mathcal{C}$3).

Next, it follows from ($\mathcal{C}$1) that
\begin{align*}
\sup_{i\in\mathbb{N}}\mathbb{E} {\rm e}^{\gamma
(Z_i-c\theta_i)}
\leqslant\sup_{i\in\mathbb{N}}\mathbb {E} {\rm
e}^{\gamma Z_{i}}<\infty. \notag
\end{align*}

So, the requirement ($\mathcal{C}1^*$) holds too.

It remains to prove that
\begin{align}
\label{amen} \lim_{u\to\infty}\sup_{i\in\mathbb{N}}
\mathbb{E} \bigl(|Z_i-c\theta_i|\ind_{\lbrace Z_i-c\theta_i\leqslant-u\rbrace} \bigr)=0.
\end{align}

To establish this, we use the inequality
\begin{align}
\sup_{i\in\mathbb{N}}\mathbb{E} \big(|Z_i-c \theta_i|\ind_{\lbrace
Z_i-c\theta_i \leqslant-u\rbrace} \big)
&{}\leqslant\sup_{i\in\mathbb{N}}\mathbb{E} \big(Z_i\ind_{\lbrace Z_i-c\theta_i\leqslant-u\rbrace}\big)
\nonumber\\
&\quad{}+c\sup_{i\in\mathbb{N}}\mathbb{E} \big(\theta_i\ind
_{\lbrace Z_i-c\theta_i\leqslant-u\rbrace}\big).\label{r:desimta}
\end{align}

Taking the limit as $u\rightarrow\infty$ in the first summand of the
right side of inequality \eqref{r:desimta}, we get
\begin{align}
\label{r:vienuolikta}
&\lim_{u\to\infty}\sup_{i\in\mathbb{N}}\mathbb{E} \big(Z_i\ind_{\lbrace Z_i-c\theta_i\leqslant-u\rbrace}\big)
\nonumber\\
&\quad{}\leqslant \lim_{u\to\infty}\sup_{i\in\mathbb{N}}\mathbb{E}
\big(Z_i\ind_{\lbrace Z_i-c\theta_i\leqslant-u\rbrace}\ind_{\lbrace\theta
_i\leqslant\frac{u}{2c}\rbrace}\big)
\notag\\&\qquad{}
+\lim_{u\to\infty}\sup_{i\in\mathbb{N}}\mathbb{E}
\big(Z_i\ind _{\lbrace Z_i-c\theta_i\leqslant-u\rbrace}\ind_{\lbrace\theta_i> \frac
{u}{2c}\rbrace}\big)\notag
\\
&\quad{}\leqslant \lim_{u\to\infty}\sup_{i\in\mathbb{N}}\mathbb{E}
\big(Z_i\ind_{\lbrace Z_i\leqslant-u/2\rbrace}\big)
\notag\\&\qquad{}
+\lim_{u\to\infty}\sup_{i\in\mathbb{N}}\mathbb{E}
\big(Z_i\ind _{\lbrace Z_i-c\theta_i\leqslant-u\rbrace}\ind_{\lbrace\theta_i>\frac
{u}{2c}\rbrace}\big)\notag
\\
&\quad{}=\lim_{u\to\infty}\sup_{i\in\mathbb{N}}\mathbb{E}
\big(Z_i\ind _{\lbrace Z_i-c\theta_i\leqslant-u\rbrace}\ind_{\lbrace\theta_i> \frac
{u}{2c}\rbrace}\big)\notag
\\
&\quad{}\leqslant\lim_{u\to\infty}\sup_{i\in\mathbb{N}}\mathbb{E}
\big(Z_i\ind_{\lbrace\theta_i> \frac{u}{2c}\rbrace} \big)\notag
\\
&\quad{}= \lim_{u\to\infty}\sup_{i\in\mathbb{N}}
\mathbb{E}Z_i\mathbb {P} \Bigl(\theta_i>
\frac{u}{2c} \Bigr)\notag
\\
&\quad{}\leqslant\sup_{i\in\mathbb{N}}\mathbb{E}Z_i\lim
_{u\to\infty}\sup_{i\in\mathbb{N}}\mathbb{P} \Bigl(
\theta_i> \frac{u}{2c} \Bigr).
\end{align}

Since
$
x\leqslant{{\rm e}^{\gamma x}}/{\gamma}$, $x\geqslant0$, condition
($\mathcal{C}$1) implies that
\begin{align}
\label{r:dvylikta} \sup_{i\in\mathbb{N}}\mathbb{E}Z_i<\infty.
\end{align}

In addition,
\begin{align}
\label{r:trylikta} \lim_{u\to\infty}\sup_{i\in\mathbb{N}}\mathbb{P}
\Bigl(\theta_i> \frac
{u}{2c} \Bigr)&=\lim
_{u\to\infty}\sup_{i\in\mathbb{N}}\mathbb{E} \biggl(
\frac{\theta_i\ind_{\lbrace\theta_i> \frac{u}{2c}\rbrace}}{\theta
_i} \biggr)\notag
\\
&\leqslant\lim_{u\to\infty}\frac{2c}{u}\sup
_{i\in\mathbb{N}}\mathbb {E} \big(\theta_i\ind_{\lbrace\theta_i> \frac{u}{2c}\rbrace}
\big){=}0
\end{align}
by condition ($\mathcal{C}$2).

Therefore, relations \eqref{r:vienuolikta}, \eqref{r:dvylikta}, and
\eqref{r:trylikta} imply that
\begin{align}
\label{r:keturiolikta} \lim_{u\to\infty}\sup_{i\in\mathbb{N}}
\mathbb{E}\big(Z_i\ind _{\lbrace Z_i-c\theta_i\leqslant-u\rbrace}\big)=0.
\end{align}

Now taking the limit as $u\rightarrow\infty$ in the second summand of
the right side of inequality \eqref{r:desimta}, by condition ($\mathcal
{C}$2) we have
\begin{align}
\label{r:penkiolikta} \lim_{u\to\infty}\sup_{i\in\mathbb{N}}\mathbb{E}
\big(\theta_i\ind _{\lbrace Z_i-c\theta_i\leqslant-u\rbrace} \big)&=\lim_{u\to\infty}\sup
_{i\in\mathbb{N}}\mathbb{E} \big(\theta_i\ind_{\lbrace\theta_i\geqslant
\frac{1}{c}(Z_i+u)\rbrace} \big)
\notag
\\
&\leqslant\lim_{u\to\infty}\sup_{i\in\mathbb{N}}\mathbb{E} \big(
\theta _i\ind_{\lbrace\theta_i\geqslant\frac{u}{c}\rbrace} \big){=}0.
\end{align}
We now see that the desired equality \eqref{amen} follows from \eqref
{r:desimta}, \eqref{r:keturiolikta}, and \eqref{r:penkiolikta}.
This means that all requirements of Lemma \ref{lll} hold for r.v.s
$Z_i-c\theta_i$, $i\in\mathbb{N}$.
\end{proof}

\section*{Acknowledgments}

We would like to thank the anonymous referee for extremely detailed and
helpful comments.

\addcontentsline{toc}{section}{References}

% structpyb loaded by ispudulyte, 2015-07-29 12:23:31

\end{document}